\documentclass[11pt]{article}
\usepackage[arrow, matrix]{xy}
\usepackage{amsmath}
\usepackage{cancel}
\usepackage{amssymb}
\usepackage{amsthm}
\usepackage[mathscr]{eucal}
\usepackage[pagewise]{lineno}%\linenumbers
\input{epsf}
\theoremstyle{definition}
\newtheorem{definition}{Definition}%[section]
\newtheorem{theorem}[definition]{Theorem}
\newtheorem{proposition}[definition]{Proposition}

\newtheorem{corollary}[definition]{Corollary}
\theoremstyle{remark}
\newtheorem{remark}[definition]{Remark}

\newcounter{enumctr}

\newcommand{\R}{\mathbb{R}}

\newcommand{\mP}{\mathbb{P}}

\newcommand{\cF}{\mathcal{F}}

\renewcommand{\phi}{\varphi}

\newcommand{\rT}{\mathrm {T}}
\begin{document}

%\linenumbers

\title{\vspace*{-10mm}
On the asymptotic behavior of solutions to time-fractional elliptic equations driven a multiplicative white noise}
\author{H.T.~Tuan\footnote{\tt httuan@math.ac.vn, \rm Institute of Mathematics, Vietnam Academy of Science and Technology, 18 Hoang Quoc Viet, 10307 Ha Noi, Viet Nam}}
\date{14/02/2020}
\maketitle
\begin{abstract}
This paper devoted to study of time-fractional elliptic equations driven a multiplicative noise. By combining the eigenfunction expansion method for symmetry elliptic operators, the variation of constant formula for strong solutions to scalar stochastic fractional differential equations, Ito's formula and establishing a new weighted norm associated with a Lyapunov-Perron operator defined from this representation of solutions, we show the asymptotic behaviour of solutions to these systems in mean square sense. As a consequence, we also prove existence, uniqueness and the convergence rate of their solutions.
\end{abstract}
{{\emph {Keywords and phases}}: Stochastic partial differential equations, Time fractional derivatives, Multiplicative white noise, Mild solution, Existence and uniqueness, Asymptotic behavior}

\section{Introduction}
Calculus (derivative and  integral) is an ideal tool to describe evolutionary processes. Typically, each evolutionary process is represented by a system of differential equations. By studying (qualitative or quantitative) solutions of equations, one can know the current state as well as predict the past or future posture of the process. However, common phenomena in life are history dependent. For these phenomena, extrapolating its posture at a future time from the past depends on both local observation and the whole past. Moreover, dependence in general is not the same at all times. Fractional calculus (fractional derivative and fractional integral) is one of the theories that came up to meet those requirements.

Fractional differential equations (equations contains fractional derivatives) are of great interest in the last four decades due to its application in describing real-world problems, such as in signal processing, in financial mathematics, in biotechnology, in image processing, in control theory, and in mathematical psychology where fractional-order systems may be used to model the behaviour of human beings, specifically, the way in which a person reacts to external influences depends on the experience he or she has made in the past.

The time-fractional diffusion equations have been introduced in Physics by Nigmatullin \cite{Nigmatullin_86} to describe super slow diffusion process in a porous medium with the structure type of fractal geometry (for example the Koch's tree). On the probabilistic point of view, Metzler and Klafter \cite{Metzler_00} have pointed out that a time-fractional diffusion equation generates a non-Markovian diffusion process with a long memory. Roman and Alemany \cite{Roman_94} have  considered continuous-time random walks on fractals and observed that the average probability density of random walks on fractals obeys a diffusion equation with a fractional time derivative asymptotically. Moreover, a time-fractional diffusion equation also is used to model a relaxation phenomena in complex viscoelastic materials, see, e.g., \cite{Ginoa_92}.

The existence of solutions to time-fractional partial differential equations has been studied by many authors. In \cite{Kochubei_04}, using Fourier transform, the authors have built a basic solution for elliptic equations with smooth coefficient and fractional derivative in time. By Galerkin method and the Yoshida approximation sequence, in \cite{Zacher_09} Zacher has proposed a way to prove existence of certain weak solutions to abstract evolutionary integro-differential equations in Hilbert spaces. Using the operator theory in functional analysis and  the eigenfunction expansion method  for symmetry elliptic operators, in \cite{Yamamoto_11} Sakamoto and Yamamoto have proved  the existence and uniqueness of the weak solution for fractional diffusion-wave equations. Recently, using a De Giorgi--Nash type estimation, in \cite{Caffarelli_16} and \cite{Zacher_13}, the authors established the existence and Holder continuity of very weak solutions for fractional parabolic equations. By proposing a definition of the Caputo derivative on a finite interval in fractional Sobolev spaces, Gorenflo, Luchko and Yamamoto  \cite{Gorenflo_15} have investigated the solutions (in the distribution sense) to time-fractional diffusion equations from the operator theoretic viewpoint. This approach is developed by I. Kim, K. Kim and S. Lim in \cite{Kim_17} and H. Dong and D. Kim in \cite{Dong_19}.

In contrast to existence theory of solutions to deterministic fractional partial differential equations, there are very few researches on stochastic fractional partial differential equations. Using the integration by parts, Ito's formula and the Parseval's identity, an $L_2$-theory for stochastic time-fractional  partial differential equations is presented in \cite{Kim_15}. By a choosing a framework for infinite dimensional stochastic integration, in \cite{Baeumer_15} Baeumer, Geissert and Kovacs have showed the unique mild solution to a class of semi-linear Volterra stochastic evolution equations is mean-$p$ Holder continuous. In \cite{Nualart_19}, Chen, Hu and Nualart have studied nonlinear stochastic time-fractional slow and fast diffusion equations. They have proven the non-negativity of the fundamental solution, existence and uniqueness of solutions together with the  moment bounds  of these  solutions. In some cases, they have obtain the sample path regularity of the solutions. Based on monotonicity techniques, in \cite{LRS_18} the authors have developed a method to solve (stochastic) evolution equations on Gelfand triples with time-fractional derivative. Recently, in \cite{Kim_19}, the authors provide an $L_p$-theory for semi-linear partial differential equations driven by a space-time white noise. 

To our knowledge, until now, almost no research on the asymptotic behavior of solutions to stochastic fractional partial differential equations has been published. Motived by this fact, this paper is devoted to study the stability in mean square sense for time-fractional elliptic  equations driven a multiplicative white noise. The paper is organised as follows. In Section 2, we recall a framework of stochastic fractional differential equations. In Section 3, we first introduce a definition of mild solution. Then we prove the existence and uniqueness of the solution and state the main result of the paper (Theorem \ref{main_result}). The proof of Theorem \ref{main_result} is based on the mean square attractivity of solutions to scalar stochastic linear fractional differential equations  (Proposition \ref{main_result_1}) which is discussed in Subsection 3.2.
\section{Fractional calculus and stochastic fractional differential equations}
We briefly recall an abstract framework of fractional calculus and stochastic fractional differential equations.
%----------------------------
% Fractional calculus
%-----------------------------

Let $\alpha\in (0,1]$, $[a,b]\subset \R$ and $x:[a,b]\rightarrow \R$ be a measurable function such that $\int_a^b|x(\tau)|\;d\tau<\infty$. The \emph{Riemann--Liouville integral operator of order $\alpha$} is defined by
\[
(I_{a+}^{\alpha}x)(t):=\frac{1}{\Gamma(\alpha)}\int_a^t(t-\tau)^{\alpha-1}x(\tau)\;d\tau,
\]
where $\Gamma(\cdot)$ is the Gamma function.
The \emph{Riemann--Liouville fractional derivative} $^{RL\!}D_{a+}^\alpha x$ of $x$ on $[0,T]$ is defined by
\[
^{RL\!}D_{a+}^\alpha x(t):=(DI_{a+}^{1-\alpha}x)(t),\quad\text{for almost}\;t\in [0,T],
\]
where $D=\frac{d}{dt}$ is the usual derivative. The \emph{Caputo fractional derivative} of $x$ on $[0,T]$ is defined by
\[
^{\!C}D^\alpha_{0+}x(t)=^{RL\!}D_{a+}^\alpha (x(t)-x(0))\quad\text{for almost}\;t\in [0,T].
\]
The Caputo fractional derivative of a $d$-dimensional vector function $x(t)=(x_1(t), \dots,x_d(t))^{\rT}$ is defined component-wise as
\[
(^{C\!}D_{a+}^\alpha x)(t):=(^{C\!}D_{a+}^\alpha x_1(t),\dots,^{C\!}D_{a+}^\alpha x_d(t))^{\rT}.
\]
Let $A\in \R^{d\times d}$ and $f:[0,\infty)\rightarrow \R^d$ is a continuous vector-valued function. As showed in \cite[p. 140]{Podlubny}, the equation with the fractional order $\alpha\in (0,1)$
\begin{align*}
^{C}D^\alpha_{0+}x(t)&=Ax(t)+f(t),\quad t>0,\\
x(0)&=x_0\in \R^d,
\end{align*}
has a unique solution $x$ on $[0,\infty)$ which have the presentation
\[
x(t)=E_{\alpha}(t^{\alpha}A) x_0+\int_0^t (t-\tau)^{\alpha -1}E_{\alpha,\alpha}((t-\tau)^{\alpha}A) f(\tau) d\tau,\quad t\geq 0,
\]
where $E_{\alpha,\beta}:\R^{d\times d}\rightarrow \R^{d\times d}$ is the Mittag-Leffler function defined by
\[
E_{\alpha,\beta}(A):=\sum_{k=0}^{\infty}\frac{A^k}{\Gamma(\alpha k+\beta)}.
\]
For more details on Mittag-Leffler functions, we refer the reader to the monograph \cite{Podlubny}.

%\subsection{Fractional stochastic differential equations}
%
%
Next, we discuss on a fractional stochastic differential equation of order $\alpha\in(\frac{1}{2},1)$ in the following form
\begin{equation}\label{MainEq_1}
^{\!C}D^{\alpha}_{0+} X(t)=A X(t) +b(t,X(t)) + \sigma(t, X(t))\frac{dW_t}{dt},\quad t> 0,
\end{equation}
where $W_t$ is a standard scalar Brownian motion on an underlying complete filtered probability space $(\Omega,\cF,\mathbb{F}:=\{\cF_t\}_{t\in [0,\infty)},\mP)$ and $b,\sigma:[0,\infty)\times \R^d\rightarrow \R^{d}$ are measurable functions satisfying the following conditions.
\begin{itemize}
  \item [(H1)] There exists $L>0$ such that for all $x,y\in\R^d$, $t\in [0,\infty)$
\begin{equation*}\label{L_cond}
  \|b(t,x)-b(t,y)\|+ \|\sigma(t,x)-\sigma(t,y)\|\leq L\|x-y\|.
\end{equation*}
\item [(H2)]  $\sigma(\cdot,0)$  is essentially bounded, i.e.,
\[
\|\sigma(\cdot,0)\|_{\infty}:=\hbox{ess\;sup}_{\tau\in [0,\infty)}\|\sigma(\tau,0)\|<\infty
\]
almost sure and $b(\cdot,0)$ is $L^2$-locally integrable, i.e., for any $T>0$
\[
\int_0^T
\|b(\tau,0)\|^2\;d\tau<\infty.
\]
\end{itemize}
For each $t\in[0,\infty)$, let $\frak X_t:=\mathbb{L}^2(\Omega,\cF_t,\mP)$ be the space of all mean square integrable functions $f:\Omega\rightarrow \R^d$ with $\|f\|_{\rm ms}:=\sqrt{\mathbb{E}\|f\|^2}$. A process $\xi:[0,\infty)\rightarrow \mathbb{L}^2(\Omega,\cF,\mP)$ is said to be $\mathbb{F}$-adapted if $\xi(t)\in \frak X_t$ for all $t\geq 0$. We now restate the notion of classical solution to \eqref{MainEq_1}, see e.g.,  \cite[p. 209]{Son_19} and \cite{Doan_18}.
% Let $\mathbb{H}^2$ be the space containing all the processes $X$ which are measurable and $\mathbb{F}$-adapted with the norm
%$$\|X\|_{\mathbb{H}^2}:=\sup_{0\le t<\infty}\|X(t)\|_{\rm ms}<\infty,$$
%i.e. $X(t)\in\frak X_t$ for all $t\in [0,\infty)$ and $\|X\|_{\mathbb{H}^2}<\infty$.
\begin{definition}[Classical solution of stochastic time-fractional differential equation]
For each $\eta\in \frak X_0$, an $\mathbb{F}$-adapted process $X$ is called a solution of \eqref{MainEq_1} with the initial condition $X(0)=\eta$ if for every $t\in [0,\infty)$ it satisfies
\begin{align}
\notag X(t)  &=  \eta+\frac{1}{\Gamma(\alpha)}
\int_0^t (t-\tau)^{\alpha -1}(A X(\tau)+b(\tau,X(\tau))) d\tau+\\
&\hspace{3cm}\frac{1}{\Gamma(\alpha)}\int_0^t (t-\tau)^{\alpha -1} \sigma(\tau,X(\tau))dW_\tau.\label{IntegrableForm}
\end{align}
\end{definition}
It was proved in \cite{Doan_18} that for any $\eta\in \mathfrak X_0$, there exists a unique solution $\phi(t,\eta)$ of \eqref{IntegrableForm}.  The following result gives a special presentation of $\phi(t,\eta)$.
\begin{theorem}[A variation of constant formula for stochastic time-fractional differential equation]\label{MainTheorem} Let $\eta\in\mathfrak X_0$ arbitrary. Then the classical solution $\phi(t,\eta)$ to \eqref{MainEq_1} with the initial condition $\phi(0,\eta)=\eta$ has the form
  \begin{align*}
\phi(t,\eta)& =  E_\alpha (t^\alpha A) \eta+ \int\limits_0^t
(t - \tau)^{\alpha  - 1} E_{\alpha ,\alpha } ((t - \tau)^{\alpha}A) b(\tau, \phi(\tau,\eta))\;d\tau\\
&\hspace{1cm}+ \int\limits_0^t
(t - \tau)^{\alpha  - 1} E_{\alpha ,\alpha } ((t - \tau)^{\alpha}A)\sigma(\phi(\tau,\eta))\;dW_\tau,\quad t\geq 0.%\label{MildForm}
\end{align*}
\end{theorem}
\begin{proof}
See \cite[Theorem 2.3]{Son_19}.
\end{proof}
As an application of the preceding theorem, we obtain an explicit representation of the solution to  stochastic linear inhomogeneous fractional differential equations.
\begin{corollary}\label{integralform}
Consider the system \eqref{MainEq_1} with the initial data $X(0)=\eta$. Assume that the coefficient functions $b$ and $\sigma$ only depend on the time variable $t$. Then the explicit solution to the problem
\begin{align*}
^{\!C}D^{\alpha}_{0+} X(t)&=A X(t) +b(t) + \sigma(t)\;\frac{dW_t}{dt},\quad t>0,\\
X(0)&=\eta,
\end{align*}
as
\begin{align*}
\phi(t,\eta)
&=E_\alpha (t^\alpha A) \eta+ \int\limits_0^t
(t - \tau)^{\alpha  - 1} E_{\alpha ,\alpha } ((t - \tau)^{\alpha}A) b(\tau)\;d\tau\\
&\hspace{3cm}
+ \int\limits_0^t
(t - \tau)^{\alpha  - 1} E_{\alpha ,\alpha } ((t - \tau)^{\alpha}A)\sigma(\tau)\;dW_\tau,\quad t\geq 0.
\end{align*}
\end{corollary}
\begin{remark}
In this paper, we consider stochastic fractional differential equations driven by a scalar Brownian motion. The solution of these equations is defined as square integrable processes. Thus, from the Ito's formula, we see that it makes sense when the fractional order of these equations belongs to the interval $(\frac{1}{2},1)$.
\end{remark}
\section{Asymptotic behavior in mean square sense of solutions to time-fractional elliptic equations driven a multiplicative white noise}
\subsection{Main result}
Let $U$ be a bounded domain in $\R^d$ with the boundary $\partial U\in C^1$. We consider the equation of the order $\alpha\in (1/2,1)$
\begin{align}\label{main_eq}
\notag \frac{\partial^{\alpha} u(t,x)}{\partial t^\alpha}&=\sum_{i,j=1}^d \partial_{x_i}(a_{ij}(x)\partial_{x_j} u(t,x))+c(x)u(t,x)\\
&\hspace{1cm}+\beta u(t,x)+\gamma u(t,x)\frac{dW_t}{dt},
\end{align}
where 
\begin{itemize}
\item[(a1)] $\beta,\gamma$ are arbitrary coefficients, $u(t,x)\in \R$ with $t\in \R_+$, $x\in \bar{U}$;
\item[(a2)] $a_{ij}\in C^1 (\bar{U})$, $a_{ij}=a_{ji}$ for all $1\leq i,j\leq d$ and there exists $\theta>0$ such that $\sum_{i,j=1}^d a_{ij}(x)\xi_i\xi_j\geq \theta \|\xi\|^2$ for all $x\in \bar{U}$, $\xi\in \R^d$;
\item[(a3)] $c\in C(\bar{U})$, $c(x)\leq 0$ for all $x\in \bar{U}$;
\item[(a4)] $(W_t)_{t\in[0,\infty)}$ is a standard scalar Brownian motion on an underlying complete filtered probability space $(\Omega,\cF,\mathbb{F}:=\{\cF_t\}_{t\in [0,\infty)},\mP)$. 
\end{itemize}
Assume that the initial condition 
\begin{equation}\label{main_eq_1}
u(0,\cdot)=f\in L^2(U)
\end{equation}
is $\cF_0$-measurable and the Dirichlet condition as
\begin{equation}\label{main_eq_2}
u(t,x)=0,\quad t\geq 0,\;x\in \partial U.
\end{equation}
In this section, we will combine the eigenfunction expansion method for symmetric elliptic operators and the stability of stochastic fractional differential equation (Theorem \ref{main_result_1} below) to obtain the asymptotic behavior in mean square sense of solutions to the time-fractional stochastic elliptic equation \eqref{main_eq} with the initial data \eqref{main_eq_1} and \eqref{main_eq_2}.\\

Let $\{e_j\}_{j=1}^\infty$ be the orthonormal basis of the elliptic operator $\mathcal{L}$ defined by
$$\mathcal{L}u=-\left(\sum_{i,j=1}^d \partial_{x_i}(a_{ij}\partial_{x_j} u)+c(x)u\right)$$ in the space $L^2(U)$ with respect to eigenvalues $0<\lambda_1< \lambda_2\leq \dots\leq \lambda_n\leq \dots,$ $\lambda_n\to \infty$ as $n\to \infty$ (see e.g., \cite[p. 335]{Evans_98}). On the $L^2(U)$, we introduce two families of operators $\{S(t)\}_{t\geq 0}$ and $\{R(t)\}_{t>0}$ defined by
\[
S(t)v:=\sum_{j=1}^\infty E_\alpha((-\lambda_j+\beta)t^\alpha)\langle v,e_j\rangle_{L^2(U)}e_j,\quad t\geq 0,\;v\in L^2(U),
\]
and
\[
R(t)v:=\sum_{j=1}^\infty t^{\alpha-1}E_{\alpha,\alpha}((-\lambda_j+\beta)t^\alpha)\langle v,e_j\rangle_{L^2(U)}e_j,\quad t>0,\;v\in L^2(U).
\]
Let $T>0$ be arbitrary. The following definition is a stochastic version of the deterministic case motivated by the variation of constant formula, see e.g., \cite[Definition 2.1]{Ke_2020}.
\begin{definition}[Mild solution of the fractional stochastic elliptic equation]\label{sol_df}
An $L^2(U)$-valued process $\{u(t)\}_{t\in [0,T]}$ is called a mild solution of the problem \eqref{main_eq}, \eqref{main_eq_1}, \eqref{main_eq_2} on the interval $[0,T]$ if 
\begin{equation*}
u(t)=S(t)f+\gamma\int_0^tR(t-s)u(s)dW_s,\quad t\in[0,T],
\end{equation*}
where the integral $\int_0^tR(t-s)u(s)dW_s$ is defined by
\[
\Big\langle\int_0^tR(t-s)u(s)dW_s,x\Big\rangle_{L^2(U)}=\int_0^t\langle R(t-s)u(s),x\rangle_{L^2(U)}dW_s
\]
for all $x\in L^2(U)$, $t\geq 0$.
\end{definition}
Denote by $\mathbb{H}_T$ the space of all $L^2(U)$-valued processes $\{u(t)\}_{t\in [0,T]}$ which are predictable and satisfy $$\sup_{t\in [0,T]}\mathbb{E}\|u(t)\|^2_{L^2(U)}<\infty.$$
This is a Banach space with the norm $\|\cdot\|_{\mathbb{H}_T}$ as
\[
\|u\|_{\mathbb{H}_T}=\sqrt{\sup_{t\in [0,T]}\mathbb{E}\|u(t)\|^2_{L^2(U)}}.
\]
The following theorem show a result on the existence and uniqueness of mild solutions to the problem \eqref{main_eq} with the initial conditions \eqref{main_eq_1}, \eqref{main_eq_2}.
\begin{theorem}[Existence and uniqueness of the mild solution to the fractional stochastic elliptic equation]\label{exist_result}
Suppose that $\beta<\lambda_1$. The system \eqref{main_eq}, \eqref{main_eq_1}, \eqref{main_eq_2} has a unique mild solution in $\mathbb{H}_T$.
\end{theorem}
\begin{proof}
On the space $\mathbb{H}_T$ we establish an operator $\mathcal{T}_f$ by
\[
\mathcal{T}_fu(t)=S(t)f+\gamma\int_0^tR(t-s)u(s)dW_s,\quad t\in (0,T],
\]
and $\mathcal{T}_fu(0)=f$. First, we prove that this operator is well-defined. Indeed, for any $t\geq 0$,
\begin{align}
\notag \|S(t)f\|^2_{L^2(U)}&
=\sum_{j=1}^\infty(S(t)f,e_j)_{L^2(U)}^2\\
\notag &=\sum_{j=1}^\infty E_\alpha((-\lambda_j+\beta)t^\alpha)^2 \langle f,e_j \rangle^2_{L^2(U)}\\
&\leq \sup_{t\geq 0}E_\alpha((-\lambda_1+\beta)t^\alpha)^2\|f\|^2_{L^2(U)},\label{t1}
\end{align}
and
\begin{align}
\notag \left\|\int_0^t R(t-s)u(s)dW_s\right\|^2_{L^2(U)}&=\sum_{j=1}^\infty \Big\langle\int_0^t R(t-s)u(s)dW_s,e_j\Big\rangle_{L^2(U)}^2\\
&\hspace{-5cm}=\sum_{j=1}^\infty \left(\int_0^t (t-s)^{\alpha-1}E_{\alpha,\alpha}((-\lambda_j+\beta)(t-s)^\alpha)\langle u(s),e_j\rangle_{L^2(U)}dW_s\right)^2.\label{t2}
\end{align}
Hence from \eqref{t2}, we have
\begin{align*}
&\mathbb{E}\left\|\int_0^t R(t-s)u(s)dW_s\right\|^2_{L^2(U)}\\
&\hspace{1cm}=\sum_{j=1}^\infty\mathbb{E} (\int_0^t (t-s)^{\alpha-1}E_{\alpha,\alpha}((-\lambda_j+\beta)(t-s)^\alpha)\langle u(s),e_j\rangle_{L^2(U)}dW_s)^2\\
&\hspace{1cm}=\sum_{j=1}^\infty (t-s)^{2\alpha-2}E_{\alpha,\alpha}((-\lambda_j+\beta)(t-s)^\alpha)^2\mathbb{E} \langle u(s),e_j\rangle_{L^2(U)}^2ds\\
&\hspace{1cm}\leq \int_0^t (t-s)^{2\alpha-2}E_{\alpha,\alpha}((-\lambda_1+\beta)(t-s)^\alpha)^2 \mathbb{E} \|u(s)\|^2_{L^2(U)}ds\\
&\hspace{1cm}\leq \int_0^\infty s^{2\alpha-2}E_{\alpha,\alpha}((-\lambda_1+\beta)s^\alpha)^2ds\sup_{t\in[0,T]}\mathbb{E} \|u(t)\|^2_{L^2(U)},
\end{align*}
which together \eqref{t1} implies that 
\[
\|\mathcal{T}_fu\|_{\mathbb{H}_T}=\sup_{t\in[0,T]}\mathbb{E}\|\mathcal{T}_fu(t)\|^2_{L^2(U)}<\infty.
\]
Note that for any $\rho>0$, the norm $\|\cdot\|_{\mathbb{H}_T}$ and the norm $\|\cdot\|_{\mathbb{H}_T,w}$ defined by $\|u\|_{\mathbb{H}_T,w}:=\sqrt{\sup_{t\in[0,T]}\exp{(-\rho t)}\mathbb{E}\|u(t)\|^2_{L^2(U)}}$ are equivalent. Next, we show that the operator $\mathcal{T}_f$ is contractive on $\mathbb{H}_T$ with respect to norm $\|\cdot\|_{\mathbb{H}_T,w}$. For any $u,v\in \mathbb{H}_T$, $t\in [0,T]$, we obtain the estimates
\begin{align}
\notag&\exp{(-\rho t)}\mathbb{E}\|\mathcal{T}_f u(t)-\mathcal{T}_f v(t)\|^2_{L^2(U)}\\
\notag&\hspace{1cm}\leq \gamma^2\int_0^t \exp{(-\rho (t-s))}(t-s)^{2\alpha-2}E_{\alpha,\alpha}((-\lambda_1+\beta)(t-s)^\alpha)^2\\
\notag &\hspace{3cm}\exp{(-\rho s)}\mathbb{E}\|u(s)-v(s)\|^2_{L^2(U)}ds\\
\notag&\hspace{1cm}\leq \gamma^2\sup_{t\geq 0}E_{\alpha,\alpha}((-\lambda_1+\beta)t^\alpha)^2\int_0^t \exp{(-\rho s)}s^{2\alpha-2}ds\\
\notag &\hspace{3cm}\sup_{t\in [0,T]}\exp{(-\rho t)}\mathbb{E}\|u(t)-v(t)\|^2_{L^2(U)}\\
\notag&\hspace{1cm}\leq \frac{\gamma^2\sup_{t\geq 0}E_{\alpha,\alpha}((-\lambda_1+\beta)t^\alpha)^2 \Gamma(2\alpha-1)}{\rho^{2\alpha-1}}\\
\notag&\hspace{3cm}\sup_{t\in [0,T]}\exp{(-\rho t)}\mathbb{E}\|u(t)-v(t)\|^2_{L^2(U)}\\
&\hspace{1cm}=\frac{\gamma^2\sup_{t\geq 0}E_{\alpha,\alpha}((-\lambda_1+\beta)t^\alpha)^2 \Gamma(2\alpha-1)}{\rho^{2\alpha-1}}\| u-v\|_{\mathbb{H}_T,w}^2.\label{t3}
\end{align}
Due to the estimate \eqref{t3}, we obtain
\[
\|\mathcal{T}_f u-\mathcal{T}_fv\|^2_{\mathbb{H}_T,w}\leq \frac{\gamma^2\sup_{t\geq 0}E_{\alpha,\alpha}((-\lambda_1+\beta)t^\alpha)^2 \Gamma(2\alpha-1)}{\rho^{2\alpha-1}}\| u-v\|^2_{\mathbb{H}_T,w}.
\]
Thus, for $\rho>0$ large enough, for example, $$\frac{\gamma^2\sup_{t\geq 0}E_{\alpha,\alpha}((-\lambda_1+\beta)t^\alpha)^2 \Gamma(2\alpha-1)}{\rho^{2\alpha-1}}<1,$$ then $\mathcal{T}_f$ is contractive in $\mathbb{H}_T$. The proof is complete.
\end{proof}
\begin{remark}
The proof of Theorem \ref{exist_result} true with any $T>0$ arbitrarily. Thus, we see that the problem  \eqref{main_eq}, \eqref{main_eq_1}, \eqref{main_eq_2} has the unique global mild solution on the $[0,\infty)$.
\end{remark}
Our main result in this paper is the following theorem. Its proof will be given at the end of this paper.
\begin{theorem}[Asymptotic behavior in mean square sense of the mild solution to fractional stochatic elliptic equation]\label{main_result}
Consider the system \eqref{main_eq}, \eqref{main_eq_1} and \eqref{main_eq_2}. Assume that $\beta<\lambda_1$. Then it has a unique global mild solution $u$ on $[0,\infty)$. Moreover, the following statements holds.
\begin{itemize}
\item[(i)] For  
\begin{equation}\label{c_1}
\gamma^2\int_0^\infty s^{2\alpha-2}E_{\alpha,\alpha}(-(\lambda_1-\beta)s^\alpha)^2 ds<1
\end{equation}
and any $\delta\in (0,1)$, we have
\begin{equation*}
\sup_{t\geq 0} t^\delta \mathbb{E} \|u(t)\|^2_{L^2(U)}<\infty.
\end{equation*}
\item[(ii)] For
\begin{equation}\label{c_2}
\gamma^2\int_0^\infty s^{2\alpha-2}E_{\alpha,\alpha}(-(\lambda_1-\beta)s^\alpha)^2 ds>1,
\end{equation}
there exists $f$ such that 
\begin{equation*}
\lim_{t\to \infty}\mathbb{E}\|u(t)\|_{L^2(U)}^2\nrightarrow 0.
\end{equation*}
\end{itemize}

\end{theorem}
\begin{remark}
The equation \eqref{main_eq} can be thought as an stochastic perturbed model of the time-fractional elliptic equation
\begin{equation}\label{origin_eq}
\frac{\partial^{\alpha} u(t,x)}{\partial t^\alpha}=\sum_{i,j=1}^d \partial_{x_i}(a_{ij}(x)\partial_{x_j} u(t,x))+c(x)u(t,x),\quad t>0.
\end{equation}
Theorem \ref{main_result}(i) shows that under small perturbations, for example $\beta<\lambda_1$ and $\gamma$ satisfies \eqref{c_1}, the asymptotic stability (in mean square sense) of \eqref{origin_eq} is guaranteed. However, if the noise is large (e.g., the condition \eqref{c_2} is satisfied), its stability will be broken. 
\end{remark}
\subsection{Asymptotic behavior in mean square sense of solutions to scalar stochastic fractional differential equations}
Let $\{u(t)\}_{t\geq 0}$ be the mild solution to the system \eqref{main_eq}, \eqref{main_eq_1} and \eqref{main_eq_2} and denote $y_j:=(u,e_j)_{L^2(U)}$, $1\leq j<\infty$, with the basis $\{e_j\}_{j=1}^\infty$ as showed above. From Definition \ref{sol_df}, $y_j$ satisfies
\begin{equation*}
y_j(t)=E_\alpha(-(\lambda_j-\beta)t^\alpha)f_j+\gamma \int_0^t (t-s)^{\alpha-1}E_{\alpha,\alpha}(-(\lambda_j-\beta)(t-s)^\alpha)y_j(s)dw_s,
\end{equation*}
$1\leq j<\infty$, $t\geq 0$. By virtue of Corollary \ref{integralform}, $y_j$ is the solution to the equation
\[
^{\!C}D^\alpha_{0+}y_j(t)=(-\lambda_j+\beta)y_j(t)+\gamma y_j(t)dW_t,\quad t>0
\]
with the initial condition $y_j(0)=\langle f,e_j\rangle_{L^2(U)}$.
Thus, to obtain the proof of Theorem \ref{main_result}, we need to know about the asymptotic behavior of solutions to scalar stochastic fractional differential equations with a multiplicative noise. The main object of this subsection is the following scalar equation with the fractional order $\alpha\in (1/2,1)$
\begin{equation}\label{se}
^{\!C}D^{\alpha}_{0+} x(t)=a x(t)+bx(t)\;\frac{dW_t}{dt},\quad x(0)=\eta,
\end{equation}
where $a,b$ are real coefficients and $\eta\in \mathfrak X_0$. We make the assumptions as follows:
\begin{itemize}
\item[(s1)] $a<0$ and $b^2\int_0^\infty s^{2\alpha-2}E_{\alpha,\alpha}(as^\alpha)^2ds<1;$
\item[(s2)] $a<0$ and $b^2\int_0^\infty s^{2\alpha-2}E_{\alpha,\alpha}(as^\alpha)^2ds>1.$
\end{itemize}
\begin{remark}
For $\alpha\in (0,1)$, $a<0$, from \cite[Lemma 3]{Tuan_14} or \cite[Lemma 5.1 (a)]{Nualart_19}, we see that
\[
\int_0^\infty s^{2\alpha-2}E_{\alpha,\alpha}(as^\alpha)^2 ds<\infty.
\]
\end{remark}
By using the variation of constant formula, the Ito's formula, a weighted norm and the Banach fixed point theorem, we obtain the convergence rate in mean square sense of solutions to \eqref{se}.
\begin{proposition}\label{main_result_1}
Consider the equation \eqref{se}. Then the statements holds.
\begin{itemize}
\item[(i)] Assume that the condition (s1) holds. Then for any $\delta\in (0,1)$
\begin{equation}\label{es_1}
\sup_{t\geq 0}t^\delta \|x(t)\|_{ms}^2<\infty.
\end{equation}
\item[(ii)] Assume that the condition (s2) is satisfied. Then 
\begin{equation}\label{es_2}
\|x(t)\|_{ms}\nrightarrow 0
\end{equation}
as $t\to \infty$.
\end{itemize}
\end{proposition}
\begin{proof}
Following \cite[Theorem 1]{Doan_18} and Theorem \ref{MainTheorem}, the equation \eqref{se} has a unique solution on $[0,\infty)$ in the form
\begin{equation}\label{mildform}
x(t)=E_\alpha(at^\alpha)\eta+\int_0^t (t-\tau)^{\alpha-1}E_{\alpha,\alpha}(a(t-\tau)^\alpha))bx(\tau)dW_\tau.
\end{equation}
From \eqref{mildform}, using Ito's formula (see e.g., \cite[p. 87]{Kloeden_92}), for all $t\in [0,\infty)$, we see that
\[
\mathbb{E}x(t)^2=E_\alpha(at^\alpha)^2 \mathbb{E}\eta^2+b^2 \int_0^t (t-\tau)^{2\alpha-2}E_{\alpha,\alpha}(a(t-\tau)^\alpha)^2 \mathbb{E}x(\tau)^2d\tau.
\]
Put $y(t)=\mathbb{E}x(t)^2$ and $y_0=\mathbb{E}\eta^2$.\\ 

\noindent (i) On the space of bounded and continuous functions $C_b([0,\infty),\R)$ we establish a functional as below. For any $\xi \in C_b([0,\infty),\R)$, we define
\begin{equation*}
\|\xi\|_w:=\sup_{t\in [0,\infty)}\alpha(t)|\xi(t)|,
\end{equation*}
where 
\begin{equation*}
\alpha(t)=
\begin{cases}
T^\delta,\quad t\in [0,T],\\
t^\delta,\quad t\geq T,
\end{cases}
\end{equation*}
with $T$ is positive coefficient and chosen later. It is obvious that the set $C_w([0,\infty),\R):=\{\xi\in C_b([0,\infty),\R):\|\xi\|_w<\infty\}$ is a Banach space with the norm $\|\cdot\|_w$. Now on $C_w([0,\infty),\R)$ we introduce an operator $\mathcal{T}_{y_0}$ as follows.
For any $\xi\in C_w([0,\infty),\R)$, we define
\[
\mathcal{T}_{y_0}\xi(t):=E_\alpha(at^\alpha)^2 y_0+b^2 \int_0^t (t-\tau)^{2\alpha-2}E_{\alpha,\alpha}(a(t-\tau)^\alpha)^2 \xi(\tau)\;d\tau
\]
for all $t\in [0,\infty)$.
This operator is contractive. Indeed, for any $\xi,\hat\xi\in C_w([0,\infty),\R)$, we have
\[
\mathcal{T}_{y_0}\xi(t)-\mathcal{T}_{y_0}\hat{\xi}(t)=b^2\int_0^t (t-\tau)^{2\alpha-2}E_{\alpha,\alpha}(a(t-\tau)^\alpha)^2 (\xi(\tau)-\hat{\xi}(\tau))\;d\tau
\]
for all $t\in [0,\infty)$. Consider the case where $t\in [0,T]$. In this case
\begin{align}
\notag \alpha(t)|\mathcal{T}_{y_0}\xi(t)-\mathcal{T}_{y_0}\hat{\xi}(t)|&\leq T^\delta b^2\int_0^t (t-\tau)^{2\alpha-2}E_{\alpha,\alpha}(a(t-\tau)^\alpha)^2 |\xi(\tau)-\hat{\xi}(\tau)|d\tau\\
\notag&\leq  b^2\int_0^t (t-\tau)^{2\alpha-2}E_{\alpha,\alpha}(a(t-\tau)^\alpha)^2 \alpha(\tau)|\xi(\tau)-\hat{\xi}(\tau)|d\tau\\
&\leq b^2 \int_0^\infty s^{2\alpha-2}E_{\alpha,\alpha}(as^\alpha)^2 ds \|\xi-\hat{\xi}\|_w.\label{es_2}
\end{align}
Next, for $t>T$, then
\begin{align}
\notag \alpha(t)|\mathcal{T}_{y_0}\xi(t)-\mathcal{T}_{y_0}\hat{\xi}(t)|&\leq t^\delta b^2\int_0^t (t-\tau)^{2\alpha-2}E_{\alpha,\alpha}(a(t-\tau)^\alpha)^2 |\xi(\tau)-\hat{\xi}(\tau)|d\tau\\
&\hspace{-2cm}\leq t^\delta b^2\int_0^t (t-\tau)^{2\alpha-2}E_{\alpha,\alpha}(a(t-\tau)^\alpha)^2 \tau^{-\delta}d\tau \|\xi-\hat{\xi}\|_w.\label{es_3}
\end{align}
Note that on the interval $[0,t/2]$,
\begin{align}
&\notag t^\delta\int_0^{t/2}(t-\tau)^{2\alpha-2}E_{\alpha,\alpha}(a(t-\tau)^\alpha)^2\tau^{-\delta}\;d\tau\\
&\notag\hspace{1cm}\leq t^\delta\int_0^{t/2}\frac{C}{(t-\tau)^{2+2\alpha}}\tau^{-\delta}d\tau\\
&\notag\hspace{1cm}\leq \frac{Ct^\delta}{(t/2)^{2\alpha+2}}\int_0^{t/2}\tau^{-\delta}d\tau\\
&\notag \hspace{1cm}\leq \frac{C2^{2\alpha+\delta+1}}{(1-\delta)t^{2\alpha+1}}\\
&\hspace{1cm}\leq \frac{C2^{2\alpha+\delta+1}}{(1-\delta)T^{2\alpha+1}},\label{es_4}
\end{align}
on the interval $[t/2,t-M]$,
\begin{align}
&\notag t^{\delta}\int_{t/2}^{t-M}(t-\tau)^{2\alpha-2}E_{\alpha,\alpha}(a(t-\tau)^\alpha)^2 \tau^{-\delta}d\tau\\
&\notag \hspace{1cm}\leq \frac{t^{\delta}}{(t/2)^{\delta}}\int_{t/2}^{t-M}\frac{C}{(t-\tau)^{2\alpha+2}}d\tau\\
&\hspace{1cm}\leq \frac{C 2^\delta}{(2\alpha+1)M^{2\alpha+1}}\label{es_5},
\end{align}
and on $[t-M,t]$,
\begin{align}
&\notag t^{\delta}\int_{t-M}^{t}(t-\tau)^{2\alpha-2}E_{\alpha,\alpha}(a(t-\tau)^\alpha)^2\tau^{-\delta}d\tau\\
&\notag \hspace{1cm}\leq \frac{t^\delta}{(t-M)^\delta}\int_{t-M}^{t}(t-\tau)^{2\alpha-2}E_{\alpha,\alpha}(a(t-\tau)^\alpha)^2 d\tau\\
& \hspace{1cm}\leq \frac{t^\delta}{(t-M)^\delta}\int_0^\infty s^{2\alpha-2}E_{\alpha,\alpha}(as^\alpha)^2 ds\label{es_6}.
\end{align}
From \eqref{es_4}, \eqref{es_5} and \eqref{es_6}, choosing $M>0$ and $T>2M$ such that for any $t>T$
\[
\frac{b^2C2^{2\alpha+\delta+1}}{(1-\delta)T^{2\alpha+1}}+\frac{Cb^22^\delta}{(2\alpha+1)M^{2\alpha+1}}+\frac{b^2t^\delta}{(t-M)^\delta}\int_0^\infty s^{2\alpha-2}E_{\alpha,\alpha}(as^\alpha)^2 ds<1.
\]
This combines with \eqref{es_2} and \eqref{es_3} showing that $\mathcal{T}_{y_0}$ is contractive on the space $C_w([0,\infty),\R)$.
On the other hand, it is easy to see that $\mathcal{T}_{y_0}$ is bounded in $C_w([0,\infty),\R)$. Hence, by Banach fixed point theorem, there exists a unique fixed point $\xi^*$ in $C_w([0,\infty),\R)$ which is also the fixed point of this operator in $C_b([0,\infty),\R)$. This implies the estimate \eqref{es_1}. The proof of this part is complete.\\

\noindent (ii) In this part we will use proof by contradiction. Let $\eta\in \mathfrak{X}_0\setminus\{0\}$. Suppose that \eqref{es_2} is not true, that is $$\lim_{t\to \infty}y(t)=0.$$
It is worth noting that $y(t)>0$ for all $t\geq 0$. By this fact there exists a increasing monotone consequence $\{t_k\}_{k=1}^\infty$ such that $0<t_1<t_2<...<t_k\to \infty$ and
\[
0<y(t_k)=\min_{s\in [0,t_k]}y(s),\quad k=1,2,\dots
\]
Thus for any $k=1,2,\dots$, we have
\begin{align*}
y(t_k)&=E_\alpha(at_k^\alpha)^2 y_0+b^2\int_0^{t_k}(t_k-s)^{2\alpha-2}E_{\alpha,\alpha}(a(t_k-s)^\alpha)^2y(s) ds\\
&\geq E_\alpha(at_k^\alpha)^2 y_0+b^2\int_0^{t_k}(t_k-s)^{2\alpha-2}E_{\alpha,\alpha}(a(t_k-s)^\alpha)^2 ds\;y(t_k)\\
&\geq E_\alpha(at_k^\alpha)^2 y_0+b^2\int_0^{t_k}s^{2\alpha-2}E_{\alpha,\alpha}(as^\alpha)^2 ds\;y(t_k)\\
&\geq b^2\int_0^{t_k}s^{2\alpha-2}E_{\alpha,\alpha}(as^\alpha)^2 ds\;y(t_k).
\end{align*}
This together the fact that
\[
b^2\int_0^{t_k}s^{2\alpha-2}E_{\alpha,\alpha}(as^\alpha)^2 ds>1 
\]
for $k$ is large enough leads to
\[
y(t_k)>y(t_k),
\]
a contradiction. The proof is complete.
\end{proof}
\begin{proof}[Proof of Theorem \ref{main_result}]
The proof is obtained directly by combining Theorem \ref{exist_result} and the same arguments as in the proof of Proposition \ref{main_result_1}. 
\end{proof}
\section*{Acknowledgement}
This research is funded by a grant from Vietnam Academy of Science and Technology under the grant DLTE.00.01-20/21.

\end{document}